\documentclass[11pt]{article}
\usepackage[T1]{fontenc}
\usepackage[utf8]{inputenc}
\usepackage{geometry}
\usepackage{amsmath}
\usepackage{amsthm}
\usepackage[dvipsnames]{xcolor} 
\usepackage[color=Purple!50!white, textwidth=22mm]{todonotes}
\usepackage[unicode=true,
 bookmarks=false,
 breaklinks=false,pdfborder={0 0 1},backref=section,colorlinks=false]
 {hyperref}
\usepackage{thmtools, thm-restate}

\makeatletter
\theoremstyle{plain}
\newtheorem{thm}{\protect\theoremname}[section]
\theoremstyle{definition}

\theoremstyle{plain}
\newtheorem{lem}[thm]{\protect\lemmaname}
\theoremstyle{plain}
\newtheorem{claim}[thm]{\protect\claimname}
\theoremstyle{plain}
\newtheorem{cor}[thm]{\protect\corollaryname}
\theoremstyle{plain}
\newtheorem{prop}[thm]{\protect\propositionname}
\theoremstyle{plain}
\newtheorem{ques}[thm]{\protect\questionname}
\usepackage{amssymb}

\makeatother

\providecommand{\propositionname}{Proposition}
\providecommand{\questionname}{Question}
\providecommand{\corollaryname}{Corollary}
\providecommand{\definitionname}{Definition}
\providecommand{\lemmaname}{Lemma}
\providecommand{\claimname}{Claim}
\providecommand{\theoremname}{Theorem}

\title{Small subsets with large sumset:\\ Beyond the Cauchy--Davenport bound} 
\author{Jacob Fox\thanks{Department of Mathematics, Stanford University, Stanford, CA 94305. Email: {\tt jacobfox@stanford.edu}. Research supported by a Packard Fellowship and by NSF Awards DMS-1800053 and DMS-2154169.} \and Sammy Luo\thanks{Department of Mathematics, Stanford University, Stanford, CA 94305. Email: {\tt sammyluo@stanford.edu}. Research supported by NSF GRFP Grant DGE-1656518.} \and Huy Tuan Pham\thanks{Department of Mathematics, Stanford University, Stanford, CA 94305. Email: {\tt huypham@stanford.edu}. Research supported by a Two Sigma Fellowship.} \and Yunkun Zhou\thanks{Department of Mathematics, Stanford University, Stanford, CA 94305. Email: {\tt yunkunzhou@stanford.edu}. Research supported by NSF GRFP Grant DGE-1656518.}}
\date{}

\begin{document}

\maketitle

\begin{abstract}
For a subset $A$ of an abelian group $G$, given its size $|A|$, its doubling $\kappa=|A+A|/|A|$, and a parameter $s$ which is small compared to $|A|$, we study the size of the largest sumset $A+A'$ that can be guaranteed for a subset $A'$ of $A$ of size at most $s$. We show that a subset $A'\subseteq A$ of size at most $s$ can be found so that $|A+A'| = \Omega(\min(\kappa^{1/3},s)|A|)$. Thus a sumset significantly larger than the Cauchy--Davenport bound can be guaranteed by a bounded size subset assuming that the doubling $\kappa$ is large. Building up on 
the same ideas, we resolve a conjecture of Bollob\'as, Leader and Tiba that for subsets $A,B$ of $\mathbb{Z}_p$ of size at most $\alpha p$ for an appropriate constant $\alpha>0$, one only needs three elements $b_1,b_2,b_3\in B$ to guarantee $|A+\{b_1,b_2,b_3\}|\ge |A|+|B|-1$. Allowing the use of larger subsets $A'$, we show that for sets $A$ of bounded doubling, one only needs a subset $A'$ with $o(|A|)$ elements to guarantee that $A+A'=A+A$. We also address another conjecture and a question raised by Bollob\'as, Leader and Tiba on high-dimensional analogs and sets 
whose sumset cannot be saturated by a bounded size subset. 
\end{abstract}

\section{Introduction}

For finite subsets $A,B$ of an abelian group, their sumset is $A+B=\{a+b:a \in A,b \in B\}$. Estimating the sizes of sumsets is a classical topic extensively studied in additive combinatorics, and has motivated the development of a wide variety of influential tools and techniques. One of the most classical results in this area is the Cauchy--Davenport bound \cite{Cau, Dav1, Dav2}, which says that for nonempty subsets $A,B$ of $\mathbb{Z}_p$, $|A+B|\ge \min(p,|A|+|B|-1)$. This strengthens the simple observation that for nonempty subsets $A,B\subseteq \mathbb{Z}$, $|A+B|\ge |A|+|B|-1$. The equality cases of the Cauchy--Davenport bound were later characterized by Vosper \cite{Vos}, who showed that either $A,B$ must be arithmetic progressions of the same common difference, or $\min(|A|,|B|)=1$, or $|A|+|B| \geq p-1$. 

Recently, in a sequence of papers, Bollob\'as, Leader and Tiba \cite{BLT, BLT3} studied several intriguing strengthenings of classical sumset estimates, including the Cauchy--Davenport bound. In particular, they showed that for nonempty subsets $A,B$ of integers with $|A|\ge |B|$, it is enough to use three elements from $B$ to achieve the sumset bound, that is, there exists a subset $B_{(3)}$ of $B$ of size at most three such that $|A+B_{(3)}| \ge |A|+|B|-1$. Similarly, they showed that for subsets $A,B$ of $\mathbb{Z}_p$ with $|A|\ge|B|$, there exists a subset $B_{(C)}$ of $B$ of constant size $C$ for which $|A+B_{(C)}|\ge \min(p,|A|+|B|-1)$. (Here and throughout the paper, we use the notation $A_{(s)}$ and $B_{(s)}$ to denote subsets of $A$ and $B$ of size at most $s$.) The results were also later shown in the ``medium-sized'' regime in \cite{BLT3}: for subsets $A,B$ of $\mathbb{Z}_p$ with $|A|=|B|=n$, there exist subsets $A'$ of $A$ and $B'$ of $B$ each of size at most $C\sqrt{n}$ for which $|A'|+|B'|\ge \min(p,|A|+|B|-1)$ for some absolute constant $C$. Bollob\'as, Leader and Tiba \cite{BLT,BLT3} pose several interesting conjectures motivated from their work. 

Sumset estimates beyond the Cauchy--Davenport bound have been extensively studied, where the behavior gets significantly more interesting. For example, while sets achieving the Cauchy--Davenport bound have particularly simple structure, for sets $A,B$ with $|A|=|B|=n$ and $|A+B|\le \lambda n$ for a constant $\lambda>2$, the structure gets considerably more complex and there is no exact characterization of $A$ and $B$. Nevertheless, we know from Freiman's Theorem \cite{Frei1, Frei2} that $A$ and $B$ are dense subsets of generalized arithmetic progressions of bounded dimension. This central result has been very influential in the further development of additive combinatorics and related areas, see for example \cite{Nat, TV}. Motivated by this theme, in this paper, we study sumset bounds beyond the Cauchy--Davenport bound that can be achieved using small subsets. We prove several positive results and present constructions which give information about the fundamental limits. 

\vspace{0.2cm}
\noindent {\bf Lower bounds beyond the Cauchy--Davenport bound.} Our first main result in the paper shows that we can indeed achieve bounds much better than the Cauchy--Davenport bound using a bounded number of elements, where the size of the sumset we can guarantee using the small subset grows with the size of the sumset $A+A$. 
\begin{restatable}{thm}{sym}\label{thm:sym}
There exists $c>0$ such that the following holds. Let $G$ be an abelian group. Let $A\subseteq G$ be nonempty and let $\kappa=\kappa(A):= |A+A|/|A|$. Then for each $s \ge 1$, there exists $A_{(s)}\subseteq A$ of size at most $s$ such that
\[
|A+A_{(s)}| \ge c\min(\kappa^{1/3},s)|A|.
\]
\end{restatable}
It is clear that $|A+A_{(s)}|\le s|A|$ and $|A+A_{(s)}| \le \kappa |A|$. In fact, we give a construction (see Proposition \ref{prop:lower-exponent}) showing that a bound better than $c\kappa^{1/1.29}|A|$ cannot hold, so the bound in Theorem \ref{thm:sym} is tight up to possibly replacing $\kappa^{1/3}$ by $\kappa^C$ for some $C\in [1/3,1/1.29]$. 

We also prove the following asymmetric version. 
\begin{restatable}{thm}{asym}\label{thm:asym} 
There exists $c>0$ such that the following holds. Let $G$ be an abelian group. Let $A,B\subseteq G$ be nonempty such that $|A|=|B|$ and let $\kappa = \kappa(A,B):=|A+B|/|A|$. Then for each $s \ge 1$, there exist $A_{(s)}\subseteq A$ and $B_{(s)}\subseteq B$ of size at most $s$ such that
\[
\max(|A+B_{(s)}|,|B+A_{(s)}|) \ge c\min(\kappa^{1/3},s)|A|.
\]
\end{restatable}
Unlike the result of \cite{BLT} achieving the Cauchy--Davenport bound, in our case, it is necessary to consider $\max(|A+B_{(s)}|,|B+A_{(s)}|)$ (see Proposition \ref{prop:need-sym}). In Proposition \ref{prop:diff-set}, we also derive an analog of Theorem~\ref{thm:sym} for difference sets, in which the power $1/3$ can be replaced by $1/2$. We further extend these results to the medium-sized regime; however, in this setting, we can only expect to improve over the Cauchy--Davenport bound in a lower order term (see Theorem \ref{thm:mid} and Proposition \ref{prop:mid}).   

To illustrate the idea, we sketch our proof of Theorem \ref{thm:sym}. We first run a greedy procedure to grow $A+A_{(s)}$, and terminate when adding any additional element does not increase $|A+A_{(s)}|$ by more than $c|A|$. When this happens, we show that for each element $a\in A\setminus A_{(s)}$, all but a $c$ fraction of sums $a+a',a'\in A$ have to lie in $A+A_{(s)}$. Using a path counting argument similar to the one employed in the proof of the Balog-Szemer\'edi-Gowers theorem \cite{BS, Gow}, we can then upper bound $|A+A|$ by $O_c(|A+A_{(s)}|^3/n^2)$, which yields the desired bound. The full proofs of Theorem \ref{thm:sym} and Theorem \ref{thm:asym} are contained in Section \ref{sec:lower-bound-sumsets}.

\vspace{0.2cm}
\noindent {\bf Application: The Cauchy--Davenport bound over $\mathbb{Z}_p$ using three elements.} The path counting argument to bound $|A+A|$ from the sum of $A$ and a set of bounded size is useful for other applications. In Section \ref{sec:CD3}, we use this idea to resolve a conjecture of Bollob\'as, Leader and Tiba (Conjecture 1 of \cite{BLT}), showing that over $\mathbb{Z}_p$, it suffices to use three elements to achieve the Cauchy--Davenport bound. 
\begin{restatable}{thm}{CD}\label{thm:p-3elm}
There exists $\alpha>0$ such that the following holds. Let $A,B$ be nonempty subsets of $\mathbb{Z}_p$ of size $n\le \alpha p$. Then there exists a subset $B_{(3)}$ of $B$ of size at most three such that $|A+B_{(3)}| \ge 2n-1$. 
\end{restatable}
The proof involves applying the path counting argument as in Section \ref{sec:lower-bound-sumsets} in several iterations to bound the size of $A'+B'$ for appropriate subsets $A'$ of $A$ and $B'$ of $B$. When dealing with sets that have bounded sumset, one can apply appropriate rectification results to connect the problem over $\mathbb{Z}_p$ with a problem over $\mathbb{Z}$. However, the main difficulty is that sumsets are not stable under addition or removal of a small number of elements, and %
even controlling sums of subsets of $A$ and $B$ that contain all but at most an $\epsilon$ fraction of elements does not yield similar control on the sum of $A$ and $B$. (We will later revisit this phenomenon when discussing Theorem \ref{thm:BLT-almost-all} and Theorem \ref{thm:saturating-sumset}.) We then need a careful argument that works directly with appropriate rectifications of subsets of $A$ and $B$ and adapts the argument inspired by the argument of \cite{BLT} over $\mathbb{Z}$ appropriately. 

\vspace{0.2cm}
\noindent {\bf High dimensional versions.} Another setting where one can expect to go beyond the Cauchy--Davenport bound is the high dimension setting, for example where we have a set $A$ in $\mathbb{Z}^d$ which is not close to a lower dimensional set in an appropriate sense. In particular, using the Freiman--Bilu Theorem \cite{Bilu, GT}, we can show that for every $d$ and $\epsilon$, there exists $t>0$ such that if $A\subseteq \mathbb{Z}^d$ and $A$ cannot be covered by a union of $t$ hyperplanes, then $|A+A| \ge (2^d-\epsilon)|A|$. In Section \ref{sec:high-dim}, we consider the analog of the results of \cite{BLT} over the high dimension setting, and show that the high-dimension analog of the Cauchy--Davenport bound can be achieved using only a bounded number of elements. 
\vspace*{-0.5cm}
\begin{restatable}{thm}{highdim}\label{thm:high-dim}
Let $\epsilon>0$ and let $d$ be a positive integer. Assume $t$ is sufficiently large in $d$ and $\epsilon$. Let $A$ be a subset of $\mathbb{Z}^d$ such that $A$ is not contained in the union of any $t$ parallel hyperplanes. Then there is a subset $A'$ of $A$ of size $O_{\epsilon,d,t}(1)$ for which $|A+A'| \ge (2^d-\epsilon)|A|$. 
\end{restatable}
In \cite[Question 3]{BLT3}, it was asked if a medium-sized version holds: for a subset $A\subseteq \mathbb{Z}^d$ that is not contained in the union of any $t$ parallel hyperplanes, does there exist $A'\subseteq A$ with $|A'|\le C\sqrt{|A|}$ and $|A'+A'|\ge (2^d-\epsilon)|A|$? It turns out that the answer to this question is negative, as shown in Proposition \ref{prop:neg-blt}. Nevertheless, the bounded size version of this, Theorem \ref{thm:high-dim}, holds. 

\vspace{0.2cm}
\noindent {\bf Saturating the entire sumset from a subset.} In the following parts, we will work with subsets of general abelian groups. As discussed earlier, it is not possible to achieve a fixed positive fraction of $|A+A|$ using the sum of $A$ and a subset $A_C$ of $A$ of size $C$ depending only on the doubling constant $\kappa=|A+A|/|A|$. %
Nevertheless, it is interesting to determine how large a subset $A'$ of $A$ is needed in order to guarantee $|A+A'|\ge (1-\epsilon)|A+A|$. It turns out that if $A$ has bounded doubling, $\kappa\le \lambda = O(1)$, one can always guarantee the existence of a subset $A'$ of $A$ of size $o(|A|)$ with $A'+A=A+A$. Furthermore, the asymmetric version holds: for any $A$ and $B$ with $|A|=|B|$ and $|A+B|\le \lambda |A|$, one can find subsets $A'$ of $A$ and $B'$ of $B$ each of size $o(|A|)$ such that $A+B=(A'+B)\cup (A+B')$. In the case where $A,B$ are dense subsets of $\mathbb{F}_p^n$%
, this was earlier shown by Ellenberg \cite{Ell} with efficient bounds building on the cap-set breakthrough. In Section \ref{sec:all}, we show that a much more general phenomenon holds for subsets of groups with bounded doubling which are not necessarily dense. 

\begin{restatable}{cor}{allsumS}\label{cor:allsum-S}
Let $G$ be an abelian group, and let $S,T\subseteq G$ be nonempty subsets of equal size such that $|S+T|=O(|S|)$. Then there exist $S'\subseteq S$ and $T'\subseteq T$ with $S+T=(S'+T)\cup (S+T')$ and $|S'|+|T'| = o(|S|)$. 
\end{restatable}

Our proof also gives rise to the appropriate generalization over general groups. For subsets $A,B$ of a group which is not necessarily abelian, we write $AB=\{ab:a \in A, b \in B\}$.

\begin{restatable}{cor}{allsumSnonab}\label{cor:allsum-S-nonab}
Let $G$ be a group, and let $S,T\subseteq G$ be nonempty subsets of equal size such that $|S^2T|=O(|S|)$. Then there exist $S'\subseteq S$ and $T'\subseteq T$ with $ST=(S'T)\cup (ST')$ and $|S'|+|T'| = o(|S|)$. 
\end{restatable}

Both corollaries follow from a more general result in Section \ref{sec:all}, which is proved using the graph theoretic triangle removal lemma. A reduction from an arithmetic removal lemma to a graph removal lemma goes back to work of Kr\'al', Serra and Vena \cite{KSV}, and our main technical result builds on and extends their argument. As applications of the graph removal lemma, our proofs of Corollaries \ref{cor:allsum-S} and \ref{cor:allsum-S-nonab} do not achieve strong quantitative bounds over general groups. By using the bounds for the arithmetic removal lemma over $\mathbb{F}_p^n$ due to Fox and Lov\'asz \cite{FL}, we can achieve a polynomial saving, replacing $o(|S|)$ by an appropriate power $|S|^{1-c_m}$ in the case where $G$ is an abelian group with bounded exponent $m$. Nevertheless, in the case $G=\mathbb{F}_p^n$, the argument of Ellenberg \cite{Ell} gives a better (in fact sharp) exponent for Corollary \ref{cor:allsum-S}. In Section \ref{sec:all}, we also give constructions in the cases $G=\mathbb{Z}$ and $G=\mathbb{F}_p^n$ suggesting fundamental limits on quantitative bounds to the corollaries. 

\vspace{0.2cm}
\noindent {\bf A construction of a non-saturating set.} While it is not possible to achieve almost all of $A+A$ using a bounded number of elements, one of the key results in \cite{BLT} says that this is possible upon replacing $A$ with a subset $A'$ consisting of almost all elements of $A$. 
\begin{thm}[\cite{BLT}]\label{thm:BLT-almost-all}
For $\lambda,\epsilon>0$, there exists $C>0$ such that the following holds. Let $A,B$ be nonempty subsets of an abelian group $G$ with $|A|=|B|$. Then there exist $A'\subseteq A$, $B'\subseteq B$ with $|A'|\ge (1-\epsilon)|A|$ and $|B'|\ge (1-\epsilon)|B|$, such that if $B''$ consists of $C$ randomly chosen elements of $B'$, then 
\[
\mathbb{E}[|A'+B''|]\ge \min((1-\epsilon)|A'+B'|,\lambda|A|,\lambda|B|).
\]
In particular, if $|A+B|\le \lambda|A|$, then there exists a subset $B''$ of $B'$ of size at most $C$ such that $|A'+B''| \ge (1-\epsilon)|A'+B'|$. 
\end{thm}
Indeed, much of the subtlety in obtaining sumset bounds growing with $|A+A|$ is that, unlike the Cauchy--Davenport bound which is stable under passing $A$ to a subset $A'$ that contains almost all elements of $A$, sumsets can generally behave unstably under addition or removal of few elements. For example, consider $A$ which is the union of a progression $P$ of length $n$, and a small collection of $k$ arbitrary elements. Then $|A+A|$ can be as large as $kn$, whereas upon removal of $k$ elements, the size of the sumset is only $2n-1$, matching the Cauchy--Davenport bound. This is the main reason behind stopping at the Cauchy--Davenport bound in \cite{BLT}, and Theorem \ref{thm:BLT-almost-all} suggests that this is the only obstruction. %

The example of a progression together with few additional elements mentioned above suggests that it is impossible to replace $|A'+B'|$ with $|A+B|$ in Theorem \ref{thm:BLT-almost-all}. However, it is unclear if we can replace only one of two sets with a subset. Considering that we are allowed to make our selection of elements from $B$, it could be the case that we only need to replace $A$ by a subset $A'$ with better additive structure.  Bollob\'as, Leader and Tiba (Question 5, \cite{BLT}) ask this explicitly: Given constants $\epsilon>0$ and $\lambda>0$, is there a constant $c>0$ such that the following assertion holds? If $A,B$ are subsets of $\mathbb{Z}_p$ with $|A|=|B|=n$ and $|A+B|\le \lambda n$, then there are sets $A'\subseteq A$ and $B'\subseteq B$ such that $|A'|\ge (1-\epsilon)n$ and $|B'|\le c$ and $|A'+B'|\ge (1-\epsilon)|A'+B|$.

It turns out that the answer to this question is negative. In Section \ref{sec:saturating-sumset}, we give a construction showing this in a strong sense: It is not possible to achieve a fixed positive fraction of the sumset $A'+A$ using a bounded number of elements from $A'$ even when we are allowed to choose $A'$ which is a subset of $A$ of size at least $(1-\epsilon)|A|$. 
\begin{restatable}{thm}{saturatingsumset}\label{thm:saturating-sumset}
For any $\nu>0$, there exist $\lambda>1$ and sufficiently small $\epsilon>0$ such that for all $s>0$ and $p$ sufficiently large in $s$ and $\nu$, there exists a subset $A$ of $\mathbb{Z}_p$ of size $\Omega_{\nu}(p)$ such that $|A+A| \le \lambda |A|$, and for all $A'\subseteq A$ of size at least $(1-\epsilon)|A|$ and all $A_{(s)} \subseteq A$ of size at most $s$, $|A'+A_{(s)}| \le \nu|A'+A|$. Furthermore, the same result holds over $\mathbb{Z}$ and $\mathbb{F}_2^{n}$. 
\end{restatable}
In fact, in our proof, we can even let $s$ grow at some explicit rate together with the size of $A$. Note that the existence of such a set $A$ is nontrivial: Since we replace $A$ with a subset $A'$ which is ``regular'' in an appropriate sense in terms of additive structure, if $|A'+A|$ is large, it cannot be due to a small number of bad elements, and we may expect to be able to select many translates of $A'$ that do not have significant overlap, leading to a large sumset. However, it turns out that this intuitive picture is not correct. Our construction involves a niveau-like set $A$ first constructed over $\mathbb{F}_2^n$, then adapted to the setting of appropriate cyclic groups and then to the integers and prime cyclic groups.

\vspace{0.2cm}
\noindent \textbf{Organization of the paper.} In Section \ref{sec:lower-bound-sumsets}, we consider bounds beyond the Cauchy--Davenport bound that are achievable using a bounded number of elements, and prove Theorems \ref{thm:sym} and \ref{thm:asym}, as well as their medium-sized analogs. In Section \ref{sec:CD3}, we prove Theorem \ref{thm:p-3elm}, which verifies the conjecture of Bollob\'as, Leader and Tiba that it suffices to use three elements to achieve the Cauchy--Davenport bound over $\mathbb{Z}_p$. In Section \ref{sec:high-dim}, we consider the high-dimension setting and prove Theorem \ref{thm:high-dim}. In Section \ref{sec:all}, we study the number of elements needed to saturate the entire sumset, and prove Corollaries \ref{cor:allsum-S} and \ref{cor:allsum-S-nonab}. In Section \ref{sec:saturating-sumset}, we give the construction proving Theorem \ref{thm:saturating-sumset}, that it is not possible to achieve a fixed positive fraction of $A'+A$ using a bounded-sized subset of $A'$, even when we are allowed to choose $A'$ as a subset consisting of almost all elements of $A$. 

\vspace{0.2cm}
\noindent \textbf{Notations and conventions.} We use the following standard asymptotic notations. We denote $f=O_P(g)$, $f\ll_P g$, or $g=\Omega_P(f)$ when there is a constant $C>0$ depending on $P$ so that $f\le Cg$, and we denote $f=o_P(g)$ or $g=\omega_P(f)$ if $f/g\to 0$ for fixed $P$. We omit floors and ceilings when they are not essential.

\section{General lower bounds on sumsets from small subsets}\label{sec:lower-bound-sumsets}

In this section, we will prove our results guaranteeing sumsets with size significantly larger than the Cauchy--Davenport bound using a small number of elements, assuming that $|A+B|$ significantly exceeds the size of the sets $A$ and $B$. In particular, we will prove Theorem \ref{thm:sym} and Theorem \ref{thm:asym}, as well as their ``medium-sized'' analogs.%

Throughout this and the following sections, we use a convenient path counting argument, which we record here. A similar argument is used in Gowers' proof of the Balog-Szemer\'edi-Gowers theorem (see \cite{Gow, TV}). Given a subset $C\subseteq A+B$, let $\Gamma_C$ be the bipartite graph on vertex sets $A,B$ such that $a\in A$, $b\in B$ are adjacent if and only if $a+b\in C$.

\begin{lem}\label{lem:path-odd}
Let $k$ be an odd positive integer. Let $A,B$ be subsets of an abelian group and $C\subseteq A+B$. Suppose that for every pair $(a,b)\in A\times B$ there are at least $w$ walks of length $k$ in $\Gamma_C$ going from $a$ to $b$. Then $|A+B|\le |C|^k/w$.
\end{lem} %
\begin{proof}
For each walk $(a_0,b_1,a_2,\dots,a_{k-1},b_k)$ of length $k$ with $a_0 = a$ and $b_k = b$, we have $a+b = (a+b_1)-(b_1+a_2)+\dots +(a_{k-1}+b)$. Thus, for each $x \in A+B$, fixing $(a,b)\in A\times B$ such that $x=a+b$, we have at least $w$ different representations $x=c_1-c_2+\cdots+c_k$ where $c_1,\dots,c_k\in C$. Since there are at most $|C|^k$ choices of $c_1,\dots,c_k$, and each choice uniquely determines $x\in A+B$ and the walk, we have that $|A+B|\cdot w\le |C|^k$, which rearranges to the desired inequality.
\end{proof}

We also have the following version of Lemma \ref{lem:path-odd} for even values of $k$, which follows from essentially the same proof. %

\begin{lem}\label{lem:path-even}
Let $k$ be an even positive integer. Let $A,B$ be subsets of an abelian group and $C\subseteq A+B$. Suppose that for every pair $(b,b')\in B\times B$ there are at least $w$ walks of length $k$ in $\Gamma_C$ going from $b$ to $b'$. Then $|B-B|\le |C|^k/w$.
\end{lem}

\subsection{Large sumsets from constant size subset}
In this section, we prove Theorems \ref{thm:sym} and \ref{thm:asym}. We recall their statements for convenience. 

\sym*

\asym*
Theorem \ref{thm:sym} is a special case of Theorem \ref{thm:asym}. Now we present the proof of  Theorem \ref{thm:asym}.
\begin{proof}
Fix $c=\frac{1}{14}$. Let $n = |A| = |B|$. We proceed by induction on $s$ to show that for every $s\geq 1$ there exist choices of $A_{(s)}$ and $B_{(s)}$ such that $|A+B_{(s)}|+|B+A_{(s)}|\geq 2c\min(\kappa^{1/3},s)n$. For $s=1$, we trivially have $|A+B_{(s)}|+|B+A_{(s)}|\geq |A|+|B|=2sn\geq 2c\min(\kappa^{1/3},s)n$ since $c\leq 1$. Now suppose the claim holds for $s-1$, so there exist $A_{(s-1)}$ and $B_{(s-1)}$ such that $|A+B_{(s-1)}|+|B+A_{(s-1)}|\geq 2c\min(\kappa^{1/3},s-1)n$. If $|A+B_{(s-1)}|+|B+A_{(s-1)}|\geq 2c\kappa^{1/3}n$, then %
we are done. So, we can assume $2c(s-1)n\leq |A+B_{(s-1)}|+|B+A_{(s-1)}|< 2c\kappa^{1/3}n$.

If there is an element $b\in B\setminus B_{(s-1)}$ such that $|(A+b)\setminus (A+B_{(s-1)})|\geq 2cn$, then %
letting $A_{(s)}=A_{(s-1)}$, $B_{(s)}=B_{(s-1)}\cup \{b\}$ yields $|A+B_{(s)}|+|B+A_{(s)}|\geq 2csn$ as desired. So we can assume that $|(A+b)\setminus (A+B_{(s-1)})|< 2cn$, and thus $|(A+b)\cap (A+B_{(s-1)})|> (1-2c)n$, for all $b\in B\setminus B_{(s-1)}$. Similarly, we can assume $|(B+a)\cap (B+A_{(s-1)})|>(1-2c)n$ for all $a\in A\setminus A_{(s-1)}$. The same inequalities are clearly true when $a\in A_{(s-1)}$ or $b\in B_{(s-1)}$, so we in fact have
\begin{equation}
    \label{eqn:bigintersect}
    |(B+a)\cap (B+A_{(s-1)})|>(1-2c)n \,\, \forall a\in A, \qquad |(A+b)\cap (A+B_{(s-1)})|>(1-2c)n \,\, \forall b\in B.
\end{equation}
Let $C= (B+A_{(s-1)})\cup (A+B_{(s-1)})$, and consider the graph $\Gamma_C$ defined as above. From \eqref{eqn:bigintersect} we know that every vertex in $A$ is adjacent to all but $2cn$ vertices of $B$, and vice versa. In particular, given any pair $(x,y)\in A\times B$, there are at least $(1-2c)n$ choices of $x'$ and $(1-2c)n$ choices of $y'$ such that $(x,y'),(y,x')\in \Gamma_C$. For each such choice of $x'$, all but at most $2cn$ of the choices of $y'$ satisfy $(x',y')\in \Gamma_C$. Thus, for every choice of $(x,y)\in A\times B$, there are at least $(1-2c)n(1-4c)n\geq (1-6c)n^2$ paths of length three from $x$ to $y$ in $\Gamma_C$. Applying Lemma~\ref{lem:path-odd} with $k=3$, $w=(1-6c)n^2$ yields
\[
|A+B|(1-6c)n^2\leq |(B+A_{(s-1)})\cup (A+B_{(s-1)})|^3<8c^3 \kappa n^3.
\]
Note that $|A+B| = \kappa n$. Simplifying the inequality above gives $\frac{8c^3}{1-6c}> 1$, which contradicts with $c=\frac{1}{14}$. Thus \eqref{eqn:bigintersect} is false. This means that there exist $A_{(s)}$ and $B_{(s)}$ with $|A+B_{(s)}|+|B+A_{(s)}|\geq 2c\min(\kappa^{1/3},s)n$. By induction, this holds for all $s\geq 1$, and we conclude that $\max(|A+B_{(s)}|,|B+A_{(s)}|)\geq c\min(\kappa^{1/3},s)n$ as desired. 
\end{proof}

Theorem~\ref{thm:sym} follows immediately by taking $B=A$. We remark that in the case $B=-A$, a similar argument yields a better dependence on $\kappa$, as follows.
\begin{prop}\label{prop:diff-set}
There exists $c>0$ such that the following holds. Let $G$ be an abelian group. Given $A\subseteq G$ such that $|A|=n$ and $|A-A|=\kappa n$. Then for each $s \ge 1$, there exists $A_{(s)}\subseteq A$ of size at most $s$ such that 
\[
|A-A_{(s)}| \ge c\min(\kappa^{1/2},s)|A|.
\]
\end{prop}
The change comes from using Lemma~\ref{lem:path-even} and counting paths of length two in $\Gamma_C$ instead of using Lemma~\ref{lem:path-odd} and counting paths of length three. In fact, the same argument shows that, given $|A-A|=\kappa n$, we can find $A_{(s)}$ such that $|A+A_{(s)}|\geq  c\min(\kappa^{1/2},s)|A|$ as well.

On the other hand, the following construction gives an upper bound on how large of a sumset we can guarantee.

\begin{prop}\label{prop:lower-exponent}
There exists $c>0$ for arbitrarily large $\kappa>0$ for which the following holds: For all positive integers $d$, there exist infinitely many positive integers $n$ such that for some $A\subset \mathbb{Z}^d$ with $|A|=n$ and $|A+A| = \kappa n$, for any $s\le n^{c/\log\kappa}$ %
and any $A_{(s)}\subseteq A$ of size at most $s$, we have
\[
|A+A_{(s)}| \leq \min(\kappa^{1/1.29}, s)|A|.
\]
\end{prop}
\begin{proof}
Fix a small constant $\epsilon\in (0,1)$. For a positive integer $d$, a perfect square $m>1$, and a real number $\delta \in (2\epsilon^{-1}m^{-1/2}, 1-m^{-1/2})$, consider $A=([\delta m] \cup B)^d$, where $B = \{k\sqrt{m}: 0\leq k <\sqrt{m}\}\cup \{m - k: 0\leq k < \sqrt{m}\}$ is a set of size $2\sqrt{m}$. In particular, $n=|A| \in [(\delta m + \sqrt{m})^d,(\delta m +2\sqrt{m})^d]$%
. It is easy to check that for $m$ sufficiently large, $B+B \supseteq [m,2m]$ and $[\delta m]+B\supseteq [m]$. Thus, $$|A+A| = (2m)^d.$$ Hence, $\kappa \geq (2/(\delta+2m^{-1/2}))^{d}\geq (2/((1+\epsilon)\delta))^d$. %

On the other hand, for any elements $a,b\in A$, we have $a+b\in [(1+\delta)m]^d$ unless there is some $i\in [d]$ for which both $a,b$ have their $i$th coordinate in $B$. Therefore, for any $a\in A$, 
\[
|(A+a)\setminus [(1+\delta)m]^d| \le d|B||A|^{(d-1)/d}=2d\sqrt{m}|A|^{(d-1)/d}.
\]
Thus, for any $s$ and any subset $A_{(s)}$ of $A$ of size $s$, we have that $|(A+A_{(s)})\setminus [(1+\delta) m]^d|\le \frac{2sd\sqrt{m}}{\delta m+\sqrt{m}}\cdot |A|\leq \frac{2s}{\sqrt{m}}\left(\frac{1+\delta}{\delta}\right)^d |A|$. Hence, $$|A+A_{(s)}|\le (1+\delta)^dm^d + \frac{2s}{\sqrt{m}}\left(\frac{1+\delta}{\delta}\right)^d|A|.$$
Thus, for $s\le \epsilon \sqrt {m}/2$, we have  $|A+A_{(s)}|/|A|\le (1+\epsilon)((1+\delta)/\delta)^{d}$. Taking $\epsilon$ sufficiently small and noting $\min_{\delta\in (0,1)} \log((1+\delta)/\delta)/\log(2/\delta) < 1/1.29$, we obtain the desired conclusion for sufficiently large $m$. Recalling that $n\le (\delta m + 2\sqrt{m})^d$, $\kappa \ge (2/((1+\epsilon)\delta))^d$, we have that 
\[
n^{c/\log \kappa} \le \exp\left(\frac{cd\log(\delta m+2\sqrt{m})}{d\log (2/((1+\epsilon)\delta))}\right) \le (\delta m)^{2c/\log(2/\delta)} \le \epsilon\sqrt{m}/2,
\]
assuming that $c>0$ is appropriately chosen. Hence, $s\le \epsilon \sqrt{m}/2$ whenever $s\le n^{c/\log \kappa}$. 
\end{proof}

We remark that in Theorem \ref{thm:asym}, we cannot guarantee that there always exists $B_{(s)}\subseteq B$ such that $|A+B_{(s)}|$ is large. 
\begin{prop}\label{prop:need-sym}
Given positive integers $s$ and $k$ and $\epsilon\in (0,1)$, for $n$ sufficiently large, there exist $A,B\subseteq \mathbb{Z}$ such that $|A|=|B|=n$, $|A+B| > kn$, and for any subset $B_{(s)}$ of $B$ of size at most $s$, $|A+B_{(s)}| \le (2+\epsilon)n$. 
\end{prop}
\begin{proof}
Let $A = [n] \cup \{n/\epsilon,2n/\epsilon,\dots,kn/\epsilon\}$ and $B = [n]$. We have $|A+B| > kn$. For any subset $B_{(s)}$ of $B$ of size at most $s$, we have that $|A+B_{(s)}|\le  2n + sk \le (2+\epsilon)n$. 
\end{proof}

\subsection{Large sumsets from medium-sized subsets}

We now consider sumsets of medium-sized subsets, where we take subsets of both $A$ and $B$ while keeping the product of their sizes small. Bollobás, Leader, and Tiba \cite{BLT3} proved the following result.
\begin{thm}[{\cite{BLT3}, Theorem 1}]
\label{thm:blt-medium}
For all $\alpha,\beta>0$ there exists $c > 0$ such that the following holds. Let $A$ and $B$ be non-empty subsets of $\mathbb{Z}_p$ with $\alpha|B| \le |A| \le \alpha^{-1}|B|$ and $|A|+|B| \le (1 - \beta)p$. Then, for any integers $1 \le c_1 \le |A|$ and $1 \le c_2 \le |B|$ such that $c_1c_2 \ge c \max(|A|, |B|)$, there exist subsets $A' \subset A$ and $B' \subset B$ of sizes $c_1$ and $c_2$ such that $|A' + B'| \ge |A| + |B| - 1$.
\end{thm}
In particular, when $|A|=|B|=n$, and we take $c_1=c_2\ge \sqrt{cn}$, this result again yields the Cauchy--Davenport lower bound of $2n-1$ using only medium-sized subsets of $A$ and $B$. We are interested in going beyond this bound. However, in this case, unlike in the case of a constant-sized subset of $B$ plus all of $A$, it turns out that only a lower-order improvement over the Cauchy--Davenport bound is possible in general.

\begin{thm}\label{thm:mid}
For all $\beta$ there exists $c'>0$ such that the following holds. Let $A,B$ be nonempty subsets of $\mathbb{Z}_p$ with $|A|=|B|=n\leq \frac{1-\beta}{2}p$ and $|A+B|=\kappa n$%
. Then for all $c\in (0,\sqrt{n}]$ there exist subsets $A_{(s)} \subset A$ and $B_{(s)} \subset B$ of size $s\leq (c'+c)\sqrt{n}$ such that
\[
|A_{(s)}+B_{(s)}|\geq 2n-1+c\min\left(\frac{1}{4} n,\left({ \kappa^{1/3}}/{2} - 2\right)\sqrt{n}\right).
\]
\end{thm}
\begin{proof}
Let $c_2=\frac{1}{4}$.  By Theorem~\ref{thm:blt-medium}, for some $c_1>0$, there exist subsets $A'\subset A$ and $B'\subset B$ with $|A'|=|B'|=c_1\sqrt{n}$ such that $|A'+B'|\geq 2n-1$. Fixing such a pair of subsets and letting $S_0:=A'+B'$, we iteratively perform the following procedure. Given $S_{i-1}$ for $1\leq i\leq \sqrt{n}$, we form $S_i\subseteq A+B$ by either taking an element $a\in A$ such that $|(a+B)\setminus S_{i-1}|\geq c_2 n$ and letting $S_i=S_{i-1}\cup (a+B)$, or taking an element $b\in B$ such that $|(b+A)\setminus S_{i-1}|\geq c_2 n$ and letting $S_i=S_{i-1}\cup (b+A)$. The procedure terminates when neither action is possible, or once $S_{i}$ is formed for some $i\geq \sqrt{n}$. Suppose the process terminates upon the formation of $S_\ell$. %
Let $A^*$ and $B^*$ be the sets of elements chosen from $A$ and $B$, respectively, in the course of the procedure, so $|A^*|,|B^*|\leq \ell< \sqrt{n}+1$ and $S_\ell = S_0\cup (A^*+B) \cup (B^*+A)$.

If $\ell\geq \sqrt{n}$, then we have $|S_\ell|\geq |S_0|+c_2\ell n\geq |S_0|+c_2 n ^{3/2}$. Otherwise, assume $\ell<\sqrt{n}$. Then by the termination condition, for all $a\in A$ and $b\in B$ we have $|(b+A)\setminus S_{\ell}|,|(a+B)\setminus S_{\ell}|< c_2 n$. Then in the graph $\Gamma_{S_\ell}$, every vertex has degree at least $(1-c_2)n$, which as before means that between every pair $(x,y)\in A\times B$ there are at least $(1-3c_2)n^2$ paths of length three. Applying Lemma~\ref{lem:path-odd} with $k=3$ and $w=(1-3c_2)n^2$ then yields $|A+B|\leq \frac{|S_\ell|^3}{(1-3c_2)n^2}$, which means that $|S_\ell|\geq ((1-3c_2)\kappa)^{1/3}n$ in this case.

Choose subsets $\tilde{A}\subset A$ and $\tilde{B}\subset B$ of size $c \sqrt{n}$ uniformly at random. By construction, each element of $A^*+B$ appears in $A^*+\tilde{B}$ with probability at least $\frac{|\tilde B|}{B}=c n^{-1/2}$, and likewise each element of $A+B^*$ appears in $\tilde A+B^*$ with probability at least $c n^{-1/2}$. Therefore, we have
\begin{align*}
    &\mathbb{E}_{\tilde A, \tilde B}[|(A'\cup A^*\cup \tilde A)+(B'\cup B^*\cup \tilde B)|] \\
    &\geq |A'+B'|+\mathbb{E}[|((A^*\cup \tilde A)+(B^*\cup \tilde B))\setminus (A'+B')|]\\
    &\geq |S_0| + \sum_{s\in ((A^*+B)\cup (A+B^*))\setminus S_0} \Pr[s\in (A^*\cup \tilde A)+(B^*\cup \tilde B)] \\
    &\geq |S_0| + c n^{-1/2} |S_\ell \setminus S_0|\\
    &\geq (1-c n^{-1/2})|S_0| + c n^{-1/2} \min(|S_0|+c_2 n^{3/2}, (1-3c_2)^{1/3} \kappa^{1/3}n)\\
    &\geq 2n-1+c \min(c_2 n,((1-3c_2)^{1/3} \kappa^{1/3}-2)n^{1/2}).
\end{align*}
Fix choices of $\tilde A,\tilde B$ such that the size of this sumset is at least its expectation, and let $A_{(s)}=A'\cup A^*\cup \tilde A$ and $B_{(s)}=B'\cup B^*\cup \tilde B$. Note that $A_{(s)},B_{(s)}$ each has size at most $s=(c'+c)\sqrt{n}$, where $c'=c_1+1$. We then have $|A_{(s)}+B_{(s)}|\geq 2n-1+c\min(\frac{1}{4} n,(c'' \kappa^{1/3}-2)\sqrt{n})$, where $c''=(1-3c_2)^{1/3}\geq \frac{1}{2}$, as desired.
\end{proof}

The matching upper bound showing that Theorem~\ref{thm:mid} is best possible up to the power of $\kappa$ comes from taking $s=C\sqrt{n}$ in the following construction.

\begin{prop}\label{prop:mid}
For all positive integers $n,k$ with $k\leq \sqrt{n}$, there exist a real number $\kappa\in (k,k+1]$ and a set $A\subseteq \mathbb{Z}$ such that $|A|=n$, $|A+A|=\kappa n$, and for all subsets $A_{(s)},B_{(s)}\subset A$ of size $s\leq n$, we have
\[
|A_{(s)}+B_{(s)}|\leq 2n+2\kappa s.
\]
\end{prop}
\begin{proof}
Let $n_0=n-k+1$, and let $A=[n_0]\cup \{2n_0,3n_0,\dots,kn_0\}$. Then $|A|=n$ and $|A+A|=|[(k+1)n_0]\cup \{(k+2)n_0,\dots,2kn_0\}|=(k+1)(n-k+1)+k-1$, so indeed $(k+1)n\geq |A+A|\geq (k+1)(n-k+1)>(k+1)n-k^2\geq kn$.

Fix arbitrary subsets $A_{(s)},B_{(s)}\subset A$ with $|A_{(s)}|=|B_{(s)}|=s$. Partition $A_{(s)}$ as $A_0\sqcup A_1$ where $A_0= A_{(s)}\cap [n_0]$, and partition $B_{(s)}=B_0\sqcup B_1$ similarly. We have $A_0+B_0\subseteq [2n_0]$ while $A_1+B_1\subseteq \{2n_0,3n_0,\dots,2kn_0\}$, and $|A_1|,|B_1|\leq |\{2n_0,3n_0,\dots,kn_0\}|=k-1$. Then
\begin{align*}
    |A_{(s)}+B_{(s)}| &\leq |A_0+B_0|+|A_0+B_1|+|A_1+B_0|+|A_1+B_1| \\
    &\leq 2n_0 + |A_0||B_1|+|A_1||B_0| +(2k-1)\\
    &\leq 2n + 1 + 2(k-1)s \leq 2n+2\kappa s,
\end{align*}
as desired.
\end{proof}

\section{The Cauchy--Davenport bound in $\mathbb{Z}_p$ from three elements}\label{sec:CD3}
In this section, we prove Theorem \ref{thm:p-3elm}, restated below, which resolves Conjecture 1 of \cite{BLT}.

\CD*

Our argument is based on an extension of the argument in Section \ref{sec:lower-bound-sumsets}, used to show that appropriate subsets of $A$ and $B$ have bounded doubling. We then use Freiman's rectification principle to relate appropriate subsets of $\mathbb{Z}_p$ and subsets of $\mathbb{Z}$. 

We will make use of the analog of the above theorem in the integer setting that was already shown in \cite{BLT}.
\begin{thm}[\cite{BLT}, Theorem 1]\label{thm:3-integer}
Let $A,B$ be subsets of $\mathbb{Z}$ of size $n$. Then there exists a subset $B_{(3)}$ of $B$ of size at most three such that $|A+B_{(3)}| \ge 2n-1$. 
\end{thm}

We will also make use of Freiman's rectification principle \cite{BLR, Frei1}. Recall that a Freiman $s$-homomorphism $\phi:A\to B$ is a map such that $\sum_{i=1}^{s}\phi(a_i)=\sum_{i=1}^{s}\phi(a_i')$ whenever $\sum_{i=1}^{s}a_i=\sum_{i=1}^{s}a_i'$, and it is called an $s$-isomorphism when $\phi^{-1}$ is also a Freiman $s$-homomorphism. 
\begin{thm}[\cite{BLR,Frei1}]\label{thm:rect}
Let $A$ be a subset of $\mathbb{Z}_p$. For $\lambda>0$ and $s>0$, there exists $c>0$ such that if $|A|<cp$ and $|A+A|\le \lambda|A|$, then there is a Freiman $s$-isomorphism between $A$ and a subset of $\mathbb{Z}$. 
\end{thm}
We use the following version of the Pl\"unnecke-Ruzsa Inequality with different summands \cite{Ruz}.
\begin{lem}[The Pl\"unnecke-Ruzsa Inequality, \cite{Ruz}]\label{lem:PR}
Let $X,Y_1,\dots,Y_k$ be subsets of an abelian group $G$. Assume that $|X+Y_i|\le \alpha_i |X|$ for each $i\le k$. Then $|Y_1+\dots +Y_k|\le \alpha_1\dots \alpha_k |X|$. 
\end{lem}
Using the Theorem \ref{thm:rect} and Lemma \ref{lem:PR}, we obtain the following corollary. 
\begin{cor}\label{cor:rect}
For $\lambda,s,t>0$, there exists $c>0$ such that the following holds. Let $A,B$ be subsets of $\mathbb{Z}_p$ with $|A|=|B|<cp$. If $|A+B|\le \lambda|A|$, then there is a Freiman $s$-isomorphism between $t(A\cup B)$ and a subset of $\mathbb{Z}$. %
\end{cor}
\begin{proof}
By Lemma \ref{lem:PR} applied with $X=A,Y_1=Y_2=B$, we have $|2B|\ll_\lambda |A|$. By another application of Lemma \ref{lem:PR} with $X = B$, we can obtain $|kB+\ell A|\ll_{k,\ell,\lambda}|A|$ and hence $|k(A\cup B)|\ll_{k,\lambda}|A|$. In particular we have $|2t(A\cup B)|\ll_{t,\lambda}|A|$, so Theorem \ref{thm:rect} implies that $t(A\cup B)$ is Freiman $s$-isomorphic to a subset $R\subseteq \mathbb{Z}$. %
\end{proof}

The following simple lemma is a warm-up to the proof of Theorem \ref{thm:p-3elm}, and suggests the general intuition behind our approach. 
\begin{lem}\label{lem:bounded-doubling}
For every $\lambda>0$, there is $\alpha>0$ such that the following holds. Let $A,B$ be subsets of $\mathbb{F}_p$ of size $n\le \alpha p$. Assume that $|A+B|\le \lambda n$. Then there exists a subset $B_{(3)}$ of $B$ of size at most three such that $|A+B_{(3)}| \ge 2n-1$. 
\end{lem}
\begin{proof}
By Corollary \ref{cor:rect}, for $\alpha$ sufficiently small, %
there is a Freiman $2$-isomorphism $\psi$ from $(A\cup B)$ to a subset of $\mathbb{Z}$. 
By Theorem \ref{thm:3-integer}, there exists a subset $B_{(3)}$ of $B$ so that $|\psi(A)+\psi(B_{(3)})| \ge 2n-1$. Noting that since $\psi$ is a Freiman $2$-isomorphism, $\psi(a)+\psi(b)=\psi(a')+\psi(b')$ if and only if $a+b=a'+b'$, this then implies $|A+B_{(3)}|=|\psi(A)+\psi(B_{(3)})|\ge 2n-1$. 
\end{proof}

The full proof of Theorem \ref{thm:p-3elm} is much more involved, due to the fact that we cannot reduce to the case $|A+B|\ll|A|$, but can only find suitable dense subsets $\tilde{A}$ of $A$ with $|\tilde{A}+B|\ll|A|$. Making the choice of elements $b_1,b_2,b_3$ greedily, we can find a subset $A'$ of $A$ of size at least $|A|/2$ and a set $C$ with $|C|\ll |A|$ such that for each $a\in A'$, at least a positive fraction of $b\in B$ satisfies $a'+b\in C$. From this we obtain the bound $|A'+B|\ll|A|$. We then bootstrap this subset to a slightly larger subset $A''$ consisting of all $a''\in A$ with $a''+b \in (A'+B)\cup A\cup (A+x)$ for a positive fraction of $b\in B$. At this point, we need to employ another argument (inspired by the argument in the integer case) in order to handle the elements of $A$ that are not contained in $A''$. To proceed with the proof in full detail, we first state two useful auxiliary results. 

\begin{lem}[Theorem 27, \cite{BLT}]\label{lem:close-interval}
There is $c>0$ such that the following holds. Let $A,B$ be subsets of $\mathbb{F}_p$ of size $n\le cp$, where $B$ is contained in an interval of length at most $n(1+c)$. Then there exists a subset $B_{(3)}$ of $B$ of size at most three such that $|A+B_{(3)}| \ge 2n-1$. 
\end{lem}

\begin{lem}[\cite{Gryn}]\label{lem:interval-diff}
There is $c>0$ such that the following holds. Let $B$ be a subset of $\mathbb{F}_p$ with $|B-B| < (2+c)|B| < 4cp$. Then B is contained in an interval of length at most $(1+c)|B|$. 
\end{lem}

\begin{proof}[Proof of Theorem \ref{thm:p-3elm}]
By Lemma \ref{lem:close-interval}, in the remaining part of the proof, we can assume that $B$ is not contained in an interval of length at most $(1+c)n$, and hence, by Lemma~\ref{lem:interval-diff}, $|B-B|\ge (2+c)n$.

\noindent \textbf{Step 1.} We first show that we can assume that there exists $B_{(2)}\subseteq B$ of size at most two where $|A+B_{(2)}| \ge (3/2+c_1)n$ for an appropriate constant $c_1>0$. 

If not, then for any two elements $b_1,b_2\in B$, we have that $|(A+b_2-b_1)\cap A|>(1/2-c_1)n$. 
By our assumption, we have that for any $b\in B-B$, we can write $b=a_1-a_2$ in at least $n(1/2-c_1)$ ways. Since the number of pairs $(a_1,a_2)$ is $n^2$, we conclude that $|B-B|\le n^2 / (n(1/2-c_1))$. For $c_1$ sufficiently small in $c$, we then have $|B-B|<(2+c)n$, a contradiction.

\noindent \textbf{Step 2.} 
Assume that for all $B_{(3)}\subseteq B$ of size at most three, $|A+B_{(3)}| \le (2+c_1/2)n$. From Step 1, there exists $B_{(2)}\subseteq B$ of size at most two with $|A+B_{(2)}| \geq (3/2+c_1)n$. Then for all $b\in B\setminus B_{(2)}$, we have $|(A+b)\setminus (A+B_{(2)})|\le (1/2+c_1/2-c_1)n=(1/2-c_1/2)n$. 

Recall the construction from Section~\ref{sec:lower-bound-sumsets} of a bipartite graph $\Gamma_C$ on vertex sets $A,B$ where $a\in A$ and $b\in B$ are connected if and only if $a+b\in C$. Let $C=A+B_{(2)}$. Then each $b\in B$ is adjacent in $\Gamma_C$ to at least $(1/2+c_1/2)n$ vertices $a\in A$. Thus, every $b,b'\in B$ are connected by at least $c_1 n$ walks of length two in the graph. By Lemma \ref{lem:path-even}, this implies that $|B-B|\ll_{c_1}n$. 

Let $A'$ %
be the set of $a\in A$ with degree at least $c_1n/4$ in the graph. The number of edges in the graph is at most $|A'|n+(n-|A'|)c_1n/4$ and at least $(1/2+c_1/2)n|B|=(1/2+c_1/2)n^2$. Hence, 
\[
|A'| \ge (1/2+c_1/4)n.
\]
Furthermore, for any $a\in A'$ and $b\in B$, there exist at least $c_1^2n^2/4$ walks %
of length three between $a$ and $b$. By Lemma \ref{lem:path-odd}, this implies $|A'+B|\ll_{c_1} n$. %
By the Pl\"{u}nnecke-Ruzsa Inequality, Lemma \ref{lem:PR}, we have $|kB|\ll_{k,c_1} n$, and $|kB-k'B+\ell A'-\ell'A'|\ll_{k,k',\ell,\ell',c_1}n$. 

Let $s$ be a sufficiently large constant to be chosen later. By Corollary \ref{cor:rect}, we have a Freiman $(s+1)$-isomorphism $\phi$ between $A'\cup B$
and a subset $R$ of $\mathbb{Z}$. By translation, without loss of generality, we can assume that $0\in A'\cup B$ and $0\in R$, $R\subset \mathbb{Z}_{\ge 0}$, and $\phi(0)=0$. Let $m=\max(\phi(B))$, and let $x=\phi^{-1}(m)\in B$. Note that there is no solution to $r_2-r_1=tm$ for $r_1,r_2\in \phi(B) \setminus \{m\}$ and $|t|\ge 1$. %
Since $\phi$ is a Freiman $(s+1)$-isomorphism on $A'\cup B$, there is no solution to $b_2-b_1=tx=t(x-0)$ for $b_1,b_2\in B \setminus \{x\}$ and $1\le |t|\le s$, since such a solution would yield a solution to $\phi(b_2)-\phi(b_1)=t(m-0)$. %

Let $S=(A'+B)\cup A\cup (A+x)$, and consider the graph $\Gamma_S$ defined as before. Let $c_2$ be a small constant (depending only on $c_1$) to be chosen later. Let $A''$ be the set of $a\in A$ with degree in $\Gamma_S$ at least $c_2n$. We have that $|A''|\ge |A'| \ge (1/2+c_1/4)n$. As before, noting that $|S|\ll_{c_1} n$, we can apply Lemma~\ref{lem:path-odd} to show that $|A''+B|\ll_{c_1,c_2} n$ (note that the implicit constant does not depend on $s$). Let $k'$ be the maximum number of elements of $A''$ contained in a progression of common difference $x$ and length at most $s$. %
For any set $\tilde{A}$ which is the intersection of $A''$ and a progression of common difference $x$ and length at most $s$, we have $|\tilde{A}+B\setminus \{x\}|=|\tilde{A}|\cdot (|B|-1)$ since there is no solution to $tx=b_1-b_2$ for $b_1,b_2\in B\setminus \{x\}$, $0<t\le s$. On the other hand, we have $|\tilde{A}+B| \ll_{c_1,c_2} n$. Hence, $k'=\max |\tilde A|\leq \frac{|\tilde A+B|}{|B|-1}\ll_{c_1,c_2}1$. %

Note that 
\begin{align}
    \mathbb{E}_{b\in B,b\ne x}[|A+\{0,x,b\}|] &= |A+\{0,x\}|+\mathbb{E}_{b\in B,b\ne x}[|(b+A)\setminus (A+\{0,x\})|] \nonumber\\
    &= |A+\{0,x\}| + \sum_{a\in A} \frac{|(a+(B\setminus \{x\}))\setminus(A+\{0,x\})|}{|B|-1}. \label{eq:count-size}
\end{align}
Let $\tilde{n}$ be the minimum number of progressions of common difference $x$ that $A$ can be partitioned into. Then $|A\cup (A+x)|=n+\tilde{n}$. %
Let $B'=B\setminus \{x\}$%
. By the definition of $A''$, we have $|(a+B')\cap(A\cup (A+x))| \le c_2n$ for $a\notin A''$. %
For $a\in A''$, we will need the following claim to bound $|(a+B')\cap (A\cup (A+x))|$.
\begin{claim}\label{claim:count-size1}
If $a\in A''$, we have $|(a+B')\cap (A\cup (A+x))|\le \tilde{n}+4(n-|A''|)/s$. 
\end{claim}
\begin{proof}
Consider a fixed translate of $B'$ by $a\in A$ and consider its intersection with a progression $P$ with common difference $x$ in $A\cup (A+x)$. If $a+b_1=p_1,a+b_2=p_2$ are in the intersection of $a+B'$ and $P$, then $p_2-p_1=b_2-b_1$. Since there is no solution to $b_2-b_1=tx$ for $1\le |t|\le s$ and $b_1,b_2\in B'$, it must be the case that $p_2-p_1=\ell x$ where either $\ell=0$ or $|\ell| > s$. In particular, in every subprogression of $P$ of length $s$, there can be at most one intersection with $a+B'$. Hence, if there are $r\ge 2$ intersections of $a+B'$ and the progression $P$, then $|P|\ge (r-1)s$. Furthermore, for $s$ chosen sufficiently large, recall that there are at most $k'=O_{c_1,c_2}(1)<s/1000$ elements in the intersection of $A''$ with any progression of length $s$ with common difference $x$. Similarly, there are at most $s/1000$ elements in $A''+x$ in a progression of common difference $x$ of length $s$. Thus, $P$ must contain at least $|P|(1-\frac{s}{500}\frac{1}{s})\ge (r-1)s/2$ elements from $(A\setminus A'')\cup((A\setminus A'')+x)$. Hence, we obtain that 
\[
|(a+B')\cap P|=r\le 1+2(|P\cap (A\setminus A'')|+|P\cap ((A\setminus A'')+x)|)/s.
\]
Summing over $P$, we obtain $|(a+B')\cap (A\cup (A+x))|\le \tilde{n}+4(n-|A''|)/s.$ 
\end{proof}

Hence, if $\tilde{n}\ge 3c_2(n-1)$, we have from (\ref{eq:count-size}) and Claim \ref{claim:count-size1} that for $s \ge 16/c_2$, %
\begin{align*}
    \mathbb{E}_{b\in B,b\ne x}[|A+\{0,x,b\}|] &\ge n+\tilde{n}+ (1-c_2\frac{n}{n-1})(n-|A''|)+\frac{|A''|}{n-1}(n-1-\tilde{n}-4(n-|A''|)/s) \\
    &= 2n -\frac{\tilde{n}}{n-1}+ (n-|A''|)\left(\frac{\tilde{n}}{n-1}-1+1-c_2\frac{n}{n-1}-\frac{4|A''|}{(n-1)s}\right)\\
    &> 2n-2.
\end{align*}
This implies the existence of $b\in B$ such that $|A+\{0,x,b\}|\ge 2n-1$. 

Next, consider the case $\tilde{n}<3c_2(n-1)$. 
We have the following claim.
\begin{claim}\label{claim:count-size2}
For $a\in A$, we have $|(a+B)\cap (A\cup (A+b))| \le 2\tilde{n} + 2n/s + 1$.  
\end{claim}
\begin{proof} 
Note that for any fixed $b\in B$, $A\cup (A+b)$ can be written as a union of at most $2\tilde{n}$ progressions of common difference $x$ with total size at most $2n$. The argument in our proof of Claim \ref{claim:count-size1} shows that for any progression $P$ of common difference $x$ and any $a\in A$, $|(a+B')\cap P|\leq 1+|P|/s$. Hence, summing over a partition of $A\cup (A+b)$ into at most $2\tilde n$ partitions, we obtain
\[
|(a+B)\cap (A\cup (A+b))| \le 2\tilde{n} + 2n/s + 1.\qedhere
\]
\end{proof}

Using Claim \ref{claim:count-size2}, we have 
\begin{align*}
    \mathbb{E}_{b,b'\in B}[|A+\{0,b,b'\}|] 
    &=|A|+\mathbb{E}_{b,b'\in B}[|(b+A)\setminus A|+|(b'+A)\setminus(A\cup (A+b))]\\
    &\ge n + 2\frac{|A|}{n}\left(n-1-2\tilde{n}-2n/s\right) >2n,
\end{align*}
assuming that $s$ is sufficiently large and $c_2$ is sufficiently small.%

\end{proof}

\section{Higher-dimensional versions}\label{sec:high-dim}
In this section, we consider bounds that can be shown in high dimension. In particular, we will show a negative answer to \cite[Question 3]{BLT3}, in that we cannot expect a much better bound than the Cauchy--Davenport bound in the medium-size regime even in high dimension. 
\begin{prop}\label{prop:neg-blt}
For all positive integers $d$, $t$, and $n$, there exists a subset $A$ of size at least $n$ in $\mathbb{Z}^d$ such that $A$ cannot be covered by a union of $t$ parallel hyperplanes, and for any $s\leq n$ and any subset $A'\subseteq A$ of size $s$, $|A'+A'| \le 2|A|+O_{d,t}(s)$. 
\end{prop}

Nevertheless, Theorem \ref{thm:high-dim}, restated below, shows that if, instead of the medium-sized version, one asks for a subset $A'$ of $A$ of bounded size with $|A'+A| \ge (2^d-\epsilon)|A|$, then such $A'$ always exists. 
\highdim*

Let $P(A;t,k,d)$ denote the property that $A$ is not contained in the union of any $t$ parallel translates of a subspace of dimension $k<d$ in $\mathbb{R}^d$. The proof of both results above will make use of the following lemma.
\begin{lem}\label{lem:const-fill-dim}
For any integers $t\geq 1$ and $k\in [0,d-1]$, there exist constants $C=C(t,k,d)$ and $T=T(t,k,d)$ such that any set $A\subseteq \mathbb{R}^d$ satisfying $P(A;T,k,d)$ has a subset $A_{(C)}$ of size at most $C$ satisfying $P(A_{(C)};t,k,d)$.
\end{lem}
\begin{proof}
We prove the result by induction on $k$. When $k=0$, the result is immediate (with $T(t,k,d)=t$ and $C(t,k,d)=t+1$). Now assume the claim holds for all $k<k_0$ and consider the case $k=k_0\in [1,d-1]$. Since $A$ satisfies $P(A;T,k,d)$, it satisfies $P(A;T',k-1,d)$ for any $T'\le T$, so by choosing $T(t,k,d)$ appropriately we can assume that $A$ satisfies $P(A;T(t^2,k-1,d),k-1,d)$. By the inductive hypothesis, we can select a subset $X$ of $A$ of size at most $C(t^2,k-1,d)$ such that $X$ satisfies $P(X;t^2,k-1,d)$. %

Suppose there are distinct subspaces $H_1,H_2$ of dimension $k$ such that $X$ is contained in both a union of $t$ translates of $H_1$ and a union of $t$ translates of $H_2$. In this case, $X$ is contained in a union of at most $t^2$ translates of $H_1\cap H_2$, which has dimension at most $k-1$, violating $P(X;t^2,k-1,d)$. Hence, there is at most one subspace $H$ of dimension $k$ such that $X$ is contained in a union of $t$ translates of $H$. If no such $H$ exists, then $X$ satisfies $P(X;t,k,d)$. Otherwise, since $A$ satisfies $P(A;T,k,d)$ with $T\geq t$, we can find a set of $t+1$ points in $A$, each in a different translate of $H$. Taking the union of $X$ with this set of $t+1$ points, we obtain a set $\bar{X}$ satisfying $P(\bar{X};t,k,d)$ with $|\bar{X}|\le C(t^2,k-1,d)+t+1$. Thus the claim holds for $k=k_0$ as well with $C(t,k,d)=C(t^2,k-1,d)+t+1$ and $T(t,k,d)=T(t^2,k-1,d)$, and hence by induction for all $k\leq d-1$. \end{proof}
Using Lemma \ref{lem:const-fill-dim}, the proofs of Proposition \ref{prop:neg-blt} and that of Theorem \ref{thm:high-dim} directly follow. 

\begin{proof}[Proof of Proposition \ref{prop:neg-blt}]
By Lemma \ref{lem:const-fill-dim} with $A = \mathbb{Z}^d$, there exists a finite subset $A_{(C)} \subset \mathbb{Z}^d$ of size $O_{t,d}(1)$ that cannot be covered by a union of $t$ parallel translates of a hyperplane. Let $A$ be the union of $A_{(C)}$ and an arithmetic progression $P$ of length $n-|A_{(C)}|$. For any subset $A' \subseteq A$ of size $s$, $A'+A'$ is contained in $(2P) \cup (A_{(C)}+A')$, so $|A'+A'|\le 2|P|+|A_{(C)}||A'|\leq 2|A|+O_{d,t}(s)$, as claimed.
\end{proof}

\begin{proof}[Proof of Theorem \ref{thm:high-dim}]

Let $\delta$ be sufficiently small in $\epsilon$. 
By Theorem \ref{thm:BLT-almost-all}, there exists $C'$ depending on $d$ and $\delta$ such that the following holds. We can find a subset $\hat{A} \subseteq A$ of size at least $(1-\delta)|A|$ such that upon choosing a random subset $\tilde{A} \subset \hat{A}$ of size $C'$, one has $\mathbb{E}[|\tilde{A}+\hat{A}|] \ge \min(2^{2d}|\hat{A}|,(1-\delta)|\hat{A}+\hat{A}|)$%
. In particular, there exists a subset $\tilde{A}$ of $\hat{A}$ of size $C'$ for which $|\tilde{A}+\hat{A}|\ge \min(2^{2d}|\hat{A}|,(1-\delta)|\hat{A}+\hat{A}|)$. If $(1-\delta)|\hat{A}+\hat{A}| \ge 2^{2d}|\hat{A}|\ge 2^{2d}(1-\delta)|A|$, then we are done; hence, we can assume that $|\tilde{A}+\hat{A}|\ge (1-\delta)|\hat{A}+\hat{A}|$. 

Let $m$ be the smallest integer for which $|\hat{A}+\hat{A}|<(2^m-\delta)|\hat{A}|$, and suppose $m\leq d$. Then by the Freiman--Bilu Theorem (\cite{Bilu, GT}), $\hat{A}$ is contained in a union of at most $K(d)$ translates of a GAP in $\mathbb{Z}^d$ of dimension at most $m-1$ and thus in a union of at most $K(d)$ translates of a subspace $Q$ of dimension at most $m-1\leq d-1$. 

Choose $t$ sufficiently large in $K(d)$. By Lemma \ref{lem:const-fill-dim}, for some $T=T(t,d)$, assuming that $A$ satisfies $P(A;T,d-1,d)$, we can select a subset $A_{(C)}$ of $A$ of size at most $C=C(t,d)$ which satisfies $P(A_{(C)};t,d-1,d)$. 
Let $A'=\tilde{A}\cup A_{(C)}$. We have $A'+A\supseteq(\tilde{A}+\hat{A})\cup (A_{(C)}+\hat{A})$. Note that $\tilde{A}+\hat{A}-\hat{A}$ is contained in a set $E_1$ which is the union of at most $K(d)^3$ translates of $Q$. Furthermore, if two translates $a+\hat{A}$ and $a'+\hat{A}$ intersect, then $a'-a \in \hat{A}-\hat{A}$, which is contained in a set $E_2$ which is the union of at most $K(d)^2$ translates of $Q$. We claim that $A_{(C)}$ contains a subset $S$ such that $S\cap E_1=(S-S)\cap E_2 =\{0\}$ and $|S|\ge (t-K(d)^3)/(2K(d)^2)$. Indeed, for
a maximal subset $S$ of $A_{(C)}$ which is disjoint from $E_1$ and with $S-S$ disjoint from $E_2\setminus \{0\}$, we have that $A_{(C)}\subseteq E_1 \cup (S+E_2)\cup (S-E_2)$, which can be covered by $2sK(d)^2+K(d)^3$ translates of $Q$. Since $A_{(C)}$ satisfies $P(A_{(C)};t,d-1,d)$, we must have $s\ge (t-K(d)^3)/(2K(d)^2)$. As such, %
\[
|(\hat{A}+A_{(C)})\setminus(\hat{A}+\tilde{A})|\ge (t-K(d)^3)|\hat{A}|/(2K(d)^2) \ge (t-K(d)^3)(1-\delta)|A|/(2K(d)^2).
\]
Hence for $t$ sufficiently large, we obtain that $|A'+A| > 2^d|A|$, as desired. 

It remains to consider the case when $m>d$, i.e. $|\hat{A}+\hat{A}|\ge (2^d-\delta)|\hat{A}|$. In this case we have
\[
|A'+A| \ge |\tilde{A}+\hat{A}|\ge (1-\delta)(2^d-\delta)(1-\delta)|A| \ge (2^d-\epsilon)|A|.\qedhere
\]
\end{proof}

\section{Saturating the entire sumset from a subset}\label{sec:all}
In this section, we study the following question. Given a finite set $A$ contained in an abelian group with $|A+A| = O(|A|)$, how large do we need $A'$ to be so that $A+A' = A+A$? We also consider the asymmetric version: Given finite sets $A,B$ contained in an abelian group with $|A|=|B|$ and $|A+B| = O(|A|)$, how large do we need $A'\subseteq A,B'\subseteq B$ to be so that $(B+A')\cup(A+B') = A+B$? We begin first with general upper bounds that hold for arbitrary groups. 

\subsection{Upper bound}
The following theorem will be our main technical result. 
\begin{thm}\label{thm:allsum}
Let $G$ be a group, let $S,T,X$ be subsets of $G$, and let $Z=XST$. Then there exist subsets $S'$ of $S$ and $T'$ of $T$ with $|S'|,|T'|=o(|Z|^2/|X|)$, and $(S'T)\cup (ST')=ST$. 
\end{thm}

We will use the graph-theoretic triangle removal lemma in the proof of Theorem \ref{thm:allsum}. The idea of using the triangle removal lemma to prove an arithmetic removal lemma in general groups goes back to work of Kr\'al', Serra and Vena \cite{KSV}. Here we build on and generalize their proof idea to prove a result that is potentially meaningful for sparse subsets of a group. We have stated below the triangle removal lemma, which follows from Szemer\'edi's regularity lemma. The best quantitative bound known for the triangle removal lemma is by Fox \cite{F}, where $1/\epsilon(\delta)$ can be taken as the inverse function of an exponential tower of height $O(\log \delta^{-1})$. 
\begin{lem}[Triangle Removal Lemma \cite{F}]
There exists $\epsilon:\mathbb{R}_+\to \mathbb{R}_+$ such that $\epsilon(\delta)\to 0$ when $\delta\to 0$ and the following holds. If $\Gamma$ is a graph with at most $\delta |V(\Gamma)|^3$ triangles, then one can remove at most $\epsilon(\delta)|V(\Gamma)|^2$ edges from $\Gamma$ and obtain a triangle free graph. 
\end{lem}
\begin{proof}[Proof of Theorem~\ref{thm:allsum}]
Let $\tau>0$ be a threshold to be chosen later. We consider a greedy procedure to define a subset $S_*$ of $S$. Start with $S_*=\emptyset$. Whenever there exists $s\in S\setminus S_*$ such that $|(sT)\setminus(S_*T)| > \tau$, we update $S_*$ to be $S_*\cup \{s\}$. The procedure terminates with a set $S_*$ satisfying the condition that for all $s\in S\setminus S_*$, $|(sT)\setminus(S_*T)| \le \tau$. Note that 
\[
|S_*|\le |ST|/\tau \le |XST|/\tau.
\]

Let $Y=XS$ and recall that $Z=XST$. Consider a tripartite graph $\Gamma$ with vertex sets $X,Y,Z$ where $x\in X$, $y\in Y$ are connected if $y=xs$ for $s\in S\setminus S_*$; $y\in Y$, $z\in Z$ are connected if $z=yt$ for $t\in T$; and $z\in Z$, $x\in X$ are connected if $z = xu$ for $u\in (ST)\setminus (S_*T)$. Each triangle in $\Gamma$ corresponds to a triple $(x,xs,xst)$ where $s\in S\setminus S_*$, $t\in T$ and $st\in (ST)\setminus (S_*T)$. For each $x\in X$ and $s\in S\setminus S_*$, the number of $t$ forming such a triple is at most $\tau$ by our choice of $S_*$. Hence, $\Gamma$ has at most $\tau|X||S|\le (\tau/|Z|)|V(\Gamma)|^3$ triangles. By the triangle removal lemma, one can remove $\epsilon(\tau/|Z|)|V(\Gamma)|^2$ edges from $\Gamma$ and ensure that no triangle remains. Let $\hat{S}$ be the set of elements $s\in S\setminus S_*$ for which at least $|X|/3$ edges $(x,xs)$ of $\Gamma$ are removed. Let $\hat{T}$ be the set of elements $t\in T$ for which at least $|X|/3$ edges $(y,yt)$ of $\Gamma$ are removed. Let $\hat{U}$ be the set of elements $u\in ST\setminus (S_*T)$ for which at least $|X|/3$ edges $(xu,x)$ are removed. For each $s \in S\setminus S_*$ and $t\in T$ with $st\in (ST)\setminus(S_*T)$, there are $|X|$ triangles $(x,xs,xst)$ in $\Gamma$. Hence, in order for all such triangles to be removed, we must have that either $s\in \hat{S}$, $t\in \hat{T}$, or $st\in \hat{U}$. For each $u\in \hat{U}$, we choose an arbitrary representation $u=s_ut_u$ for $s_u\in S,t_u\in T$. Then we have $ST=(S'T)\cup(ST')$ for $T'=\hat{T}$ and $S'=S_*\cup \hat{S}\cup \{s_u:u\in \hat{U}\}$. We also have the bound 
\[
|\hat{S}|+|\hat{T}|+|\hat{U}| \le \epsilon(\tau /|Z|) (3|Z|)^2 / (|X|/3).
\]
Thus, 
\[
|S'|+|T'|\le |Z|/\tau + 27\epsilon(\tau/|Z|)|Z|^2/|X|. 
\]
By choosing $\tau$ such that $\tau=o(|Z|)$ and $\tau = \omega(1)$, we obtain the desired bound $|S'|+|T'|=o(|Z|^2/|X|)$, noting that $|Z|\ge |X|$. 
\end{proof}

Applying Theorem \ref{thm:allsum} with $X=Y=Z=G$, we obtain the following corollary. 
\begin{cor}\label{cor:allsum-group}
Let $G$ be a group, and let $S,T\subseteq G$. Then there exist $S'\subseteq S$ and $T'\subseteq T$ with $ST=(S'T)\cup (ST')$ and $|S'|+|T'|=o(|G|)$. 
\end{cor}

Combining with the Pl\"unnecke-Ruzsa Inequality (Lemma \ref{lem:PR}), we can prove Corollary~\ref{cor:allsum-S}, restated here, which is a local version of Corollary \ref{cor:allsum-group} in abelian groups.
\allsumS*
\begin{proof}
Let $X=S$, $Y=2S$ and $Z=2S+T$. By Lemma \ref{lem:PR}, $|X|,|Y|,|Z|=\Theta(|S|)$. The result then follows from Theorem \ref{thm:allsum}. 
\end{proof}

Corollary \ref{cor:allsum-S-nonab}, the nonabelian analog of Corollary \ref{cor:allsum-S}, also immediately follows from Theorem \ref{thm:allsum} by taking $X=S$.
\allsumSnonab*

Due to the use of the triangle removal lemma, Theorem~\ref{thm:allsum} does not come with good quantitative bounds. While we do not have strong quantitative bounds for the arithmetic removal lemma (the arithmetic analog of the triangle removal lemma) for general groups, we do have them over $\mathbb{F}_p^n$ (or more generally, groups with bounded exponent). In particular, Fox and Lov\'asz \cite{FL} showed an arithmetic removal lemma with polynomial dependency over $\mathbb{F}_p^n$, stated below. Let $c_p$ be a positive constant given by $p^{1-c_p} = \inf_{0<x<1}x^{-(p-1)/3}(1+x+\dots+x^{p-1})$, and let $C_p = 1+1/c_p$.
\begin{thm}[\cite{FL}]\label{thm:FL}
Let $X,Y,Z\subseteq \mathbb{F}_p^n$ be such that there are at most $\delta p^{2n}$ solutions to $x+y=z$ for $x\in X,y\in Y,z\in Z$. Then we can remove $\epsilon p^n$ elements from $X,Y,Z$ and remove all solutions to $x+y=z$, where $\epsilon=\delta^{1/C_p-o(1)}$.  
\end{thm}
In fact, the dependency of $\delta$ on $\epsilon$ in Theorem \ref{thm:FL} is tight up to the $o(1)$ term in the exponent. 

Note that in the proof of Theorem \ref{thm:allsum} in the case $G=X=Y=Z$, denoting by $\epsilon_G(\delta)$ the bound we have for the arithmetic removal lemma over $G$, we obtain a bound
\[
|S'|+|T'|\le \inf_{\tau \ge 1} \left(|G|/\tau + 27\epsilon(\tau/|G|)|G|\right).
\]
If $\epsilon(\delta)$ depends polynomially in $\delta$, so that $\epsilon(\delta) \le \delta^c$ for some $c>0$, then we can optimize $\tau$ to obtain $|S'|+|T'| \le |G|^{1-c'}$ for some $c'>0$ depending on $c$. In particular, in the case $G=\mathbb{F}_p^n$, we can use the arithmetic removal lemma of Fox and Lov\'asz to obtain
\[
|S'|+|T'|\le |G|^{1-1/(1+C_p)+o(1)} = |G|^{1-1/(2+1/c_p)+o(1)}. 
\]
A similar result also holds more generally over abelian groups with bounded exponent. 

For $S,T\subseteq G$ of equal size with $|S+T|\le \lambda |S|$, by a result of Green and Ruzsa \cite{GR}, we can guarantee that $S$ and $T$ are subsets of a coset progression of size $O_{\lambda}(|S|)$. When $G=\mathbb{F}_p^{n}$, one can then guarantee that $S,T$ are subsets of subgroups of $G$ of size $O_{\lambda,p}(|S|)$. Hence, we also obtain the same result where we can guarantee $|S'|+|T'|\le O_{\lambda, p}(|S|^{1-1/(2+1/c_p)+o(1)})$. Furthermore, a power-saving bound holds in the same setting over abelian groups with bounded exponent. 

In \cite{Ell}, by adapting directly the linear-algebraic method behind the cap-set result, Ellenberg obtained a power-saving bound with tight exponent for the special case $G=\mathbb{F}_p^n$. Combining this result with the above argument, one obtains the following corollary.
\begin{cor}
Let $G=\mathbb{F}_p^n$. For $S,T\subseteq G$ of equal size with $|S+T| = O(|S|)$, there exist $S'\subseteq S$ and $T'\subseteq T$ with $S+T=(S'+T)\cup(S+T')$ and $|S'|+|T'| = O(|S|^{1-c_p})$. 
\end{cor}

\subsection{Lower bound construction}
In this section, we study constructions of sets $S,T$ for which $|S+T|\ll |S|=|T|$, and any choice of $S'\subseteq S$, $T'\subseteq T$ such that $S+T=(S'+T)\cup(S+T')$ must necessarily have $|S'|+|T'|$ large. In fact, we will only need to focus on the symmetric case where $S=T=A$. 

First, we will show that one cannot expect a power-saving upper bound for $|S'|+|T'|$ over $\mathbb{Z}$. We will make use of a construction of Behrend for progression-free sets \cite{Beh, Elkin, GW}.
\begin{prop}\label{prop:Behrend}
For any positive integer $t$, the set $X_{r, n}(t) := \{(x_1, \dots, x_n):1\leq x_i\leq r\;\forall\;1\leq i\leq n, \sum_{i=1}^dx_i^2 = t\}$ is progression-free in $\mathbb{Z}_m^d$ when $m \geq 2r$.
\end{prop}
\begin{proof}
Suppose that three distinct elements $x, y, z\in X_{r, n}(t)$ form a 3-term progression in $\mathbb{Z}_m^n$.

For each $i$, we know that $2y_i \equiv x_i+z_i\pmod m$. Moreover, because $1\leq x_i, y_i, z_i\leq r$, we know that $2\leq 2y_i\leq 2r\leq m$ and $2\leq x_i+z_i\leq 2r\leq m$. Therefore, $2y_i = x_i+z_i$, and hence we have
\[
    0 = 2t+2t-4t = \sum_{i=1}^n2x_i^2 + 2z_i^2 - 4y_i^2  = \sum_{i=1}^n 2x_i^2 + 2z_i^2 - (x_i+z_i)^2  = \sum_{i=1}^n(x_i-z_i)^2.
\]
This is only possible when $x = z$. Since $2y = x+z$, we have $y = (x+z)/2 = x$, which contradicts the assumption of distinctness. Therefore we conclude that $X_{r, n}(t)$ is progression-free in $\mathbb{Z}_m^n$.
\end{proof}

\begin{prop}[Construction in $\mathbb{Z}$]
For any sufficiently large positive integer $N$, there exists a finite set $A
\subset \mathbb{Z}$ such that $|A| \leq N$, $|A+A| \leq 6|A|$, and for any subset $A'\subseteq A$ with $A'+A = A+A$, we have $|A'| \geq Ne^{-c\sqrt{\log N}}$ for some absolute constant $c > 0$.
\end{prop}
\begin{proof}
Let $r, n$ be determined later. 
We would like to use Proposition~\ref{prop:Behrend}. By the Pigeonhole Principle, since $\sum_{i=1}^nx_i^2$ is in $[n, nr^2]$ when $(x_1, \dots, x_n)$ is taken from $\{1, \dots, r\}^n$, there exists $t$ such that $X_{r, n}(t)$ contains at least $\frac{r^n}{nr^2}$ elements. Fix such a $t$. Let $\phi:\mathbb{Z}_{2r}^n\to \{0,\dots,(2r)^n-1\}$ %
be a bijective map that induces a Freiman $2$-isomorphism from $\{1,\dots,r\}^n$ to its image. Since $X_{r,n}(t)$ is a progression-free set in $\mathbb{Z}^n_{2r}$, the set $A_0 := \phi(X_{r, n}(t))$ is a progression-free set in $\mathbb{Z}$. Thus, if $2a = a_1 + a_2$ for $a, a_1, a_2\in A_0$, then $a_1 = a_2 = a$.

For simplicity we denote $T = (2r)^n$. Let $A_1 = [-2T, -T-1]\cap \mathbb{Z}$ and let $A = A_0\cup A_1$. Clearly $T= |A_1| \leq |A| = |A_0| + |A_1| \leq 2T$. On the other hand, we have $A+A \subseteq [-4T, 2T)$, so $|A+A| \leq 6|A|$. Suppose $A'\subseteq A$ satisfies that $A' + A = A+A$. For any $a\in A_0$, we know that $2a\in A+A$. If $2a = a_1 + a_2$ for some $a_1, a_2\in A$, then since $a_1 + a_2 = 2a > 0$, neither $a_1$ nor $a_2$ is in $A_1$, implying that $a_1, a_2\in A_0$. From our construction of $A_0$, we must have $a_1 = a_2 = a$. Hence if $2a\in A'+A$ for some subset $A'\subseteq A$, then we must have $a\in A'$. Hence, $A+A' = A+A$ implies that $A_0\subseteq A'$. Now if we optimize by choosing $N = 2T$, $n = \sqrt{\log N}$ and $r = e^{\sqrt{\log N}}/4$, then $|A| \leq N$ and $|A'| \geq |A_0| \geq \frac{r^n}{nr^2} \geq Ne^{-c\sqrt{\log N}}$ for some positive constant $c$.
\end{proof}

As we saw in the previous subsection, the behavior over $\mathbb{F}_p^{n}$ is different, and we can show a power-saving upper bound on $|S'|+|T'|$. For the construction over $\mathbb{F}_p^n$, we will use the result of Kleinberg, Sawin and Speyer \cite{KSS} (combined with a conjecture later resolved by Pebody \cite{P} and Norin \cite{N}) on tricolored sum-free sets. 
\begin{thm}[\cite{KSS, N, P}]\label{thm:KSS}
Given a collection of ordered triples $\{(x_i, y_i, z_i)\}_{i=1}^{m}$ in $\mathbb{F}_p^n$ such that
$x_i + y_j + z_k = 0$ holds if and only if $i = j = k$, the size of the collection satisfies the
bound
$m \leq p^{(1-c_p)n}$.
Furthermore, there exists such a collection with $m \ge p^{(1-c_p)n-o(n)}$.
\end{thm}

\begin{prop}[Construction in $\mathbb{F}_p^n$]
For positive integers $n$, there exists $A\subseteq \mathbb{F}_p^n$ such that $|A+A| \leq 6|A|$, and for any subset $A'\subseteq A$ with $A'+A = A+A$, we have $|A'|\geq p^{(1-c_p)n-o(n)}$.
\end{prop}
\begin{proof}
By Theorem \ref{thm:KSS}, there exist $m\geq p^{(1-c_p)n-o(n)}$ and $\{(x_i,y_i,z_i)\}_{i=1}^m$, where $x_i, y_i, z_i\in \mathbb{F}_p^{n-2}$ for each $1\leq i\leq m$, such that $x_i+y_j+z_k = 0$ if and only if $i = j = k$. By this condition, if $x_i = x_j$, then as $x_j+y_i+z_i = 0$, we have $i = j$. Thus the $x_i$'s are distinct, so $X = \{x_i:1\leq i\leq m\}$ is of size $m$. Similarly we have that $Y = \{y_i:1\leq i\leq m\}$ is of size $m$.

For simplicity, for $a, b\in \mathbb{F}_p$ and $x\in \mathbb{F}_p^{n-2}$, we write $(a, b, x) = (a, b, x_1, \dots, x_{n-2})$ to denote an element in $\mathbb{F}_p^n$. Then we shall take $A = (\{(0, 0)\}\times X)\cup (\{(0, 1)\}\times Y) \cup (\{(1, 0)\}\times \mathbb{F}_p^{n-2})$. Clearly $A\supset \{(1, 0)\}\times \mathbb{F}_p^{n-2}$ while $A+A\subseteq \{(0, 0), (0, 1), (0, 2), (1, 0), (1, 1), (2, 0)\}\times \mathbb{F}_p^{n-2}$, so $|A+A|\leq 6|A|$.

We now show that for $A'\subseteq A$, if $A'+A = A+A$, then $|A'| \geq m$. Note that $(0, 1, -z_i) = (0, 0, x_i) + (0, 1, y_i)\in A+A$. We show that $(0, 1, -z_i)$ can only be represented in this way as a sum of two elements in $A$. Suppose that $u+v = (0, 1, -z_i)$. If both of $u, v$ are in $\{(1, 0)\}\times \mathbb{F}_p^{n-2}$, then the second coordinate is $0$, which is different from $1$ in $(0, 1, -z_i)$. If exactly one of $u, v$ is in $\{(1, 0)\}\times \mathbb{F}_p^{n-2}$, then $u+v$ must have the first coordinate $1$, which is different from $0$ in $(0, 1, -z_i)$. Finally, the only possibility remaining is that $u, v$ are both from $(\{(0, 0)\}\times X)\cup (\{(0, 1)\}\times Y)$. By considering the second coordinate in $u+v = (0, 1, -z_i)$, the only possible combination is that one of $u, v$ is from $\{(0, 0)\}\times X$ and the other one from $\{(0, 1)\}\times Y$. Without loss of generality, say that $u = (0, 0, x_j)$ and $v = (0, 1, y_k)$. Then this implies that $x_j+y_k = -z_i$, which is only possible when $i=j=k$. Hence $(0, 1, -z_i) = (0, 0, x_i) + (0, 1, y_i)$ is the only representation (up to permutation of the summands). Therefore, for each $1\leq i\leq m$, either $(0, 0, x_i)$ or $(0, 1, y_i)$ is in $A'$ if $A'+A \ni (0, 1, -z_i)$. Hence $|A'|\geq m$ as desired.
\end{proof}

\section{Non-saturating sets}\label{sec:saturating-sumset}
In this section, we give the construction proving Theorem \ref{thm:saturating-sumset}, restated below for convenience.
\saturatingsumset*

The construction is first done over $\mathbb{F}_2^n$, and is then used to derive a construction over appropriate cyclic groups $\mathbb{Z}_N$, then over the integers $\mathbb{Z}$ and prime cyclic groups $\mathbb{Z}_p$.

\noindent
\textbf{Construction over $\mathbb{F}_2^n$.} %
Let $\theta$ be a sufficiently large constant, and $\theta/4<\delta<\theta/3$. Let $m$ be a sufficiently large positive integer, and let $p$ be a sufficiently large integer so that $\nu 2^p > 100$. 

For $x\in \mathbb{F}_2^m$, let $|x|_1$ denote the number of coordinates of $x$ that are equal to $1$. Consider $A_0\subseteq \mathbb{F}_2^m$ given by $A_0 = \{x:|x|_1 \le m/2-\theta \sqrt{m}\}$. By the Central Limit Theorem, we have $|A_0| \ge \exp(-C\theta^2)\cdot 2^m$, so $A_0$ is a dense %
subset of $\mathbb{F}_2^{m}$. Consider $B_0 = \{x:|x|_1 \le \theta \sqrt{m}\}$%
. For $x,y\in \mathbb{F}_2^m$, let $x\otimes y = (x_1y_1,\dots,x_my_m)$. For any $x$ with $|x|_1\le \theta \sqrt{m}$, note that $(A_0+x)\cap \{y:|y|_1\in [m/2-\delta \sqrt{m},m/2]\} \subseteq \{y:|y\otimes x|_1 \ge (\theta -\delta)\sqrt{m}\}$. 

\begin{claim}
There is an absolute constant $c>0$ such that the following holds. For any fixed $x$ with $|x|_1 \le \theta\sqrt{m}$, the number of $y\in \{y:|y|_1\in [m/2-\delta\sqrt{m},m/2]\}$ with $|y\otimes x|_1\ge (\theta -\delta)\sqrt{m}$ is at most $2^{-c\theta \sqrt{m}}\cdot 2^m$. %
\end{claim}
\begin{proof}
Fix $x$ with $|x|_1 \le \theta\sqrt{m}$. Consider $y$ chosen uniformly at random from $\{0,1\}^{m}$. Then $|y\otimes x|_1$ is a binomial random variable with mean $\theta \sqrt{m}/2$. By the Chernoff bound, the probability that $|y\otimes x|_1 \ge (\theta-\delta)\sqrt{m}$ is at most $2^{-c\theta\sqrt{m}}$ for some absolute constant $c$ (assuming $\theta$ is sufficiently large and $\delta \in (\theta/4,\theta/3)$). Thus, the number of $y$ with $|y\otimes x|_1 \ge (\theta-\delta)\sqrt{m}$ is at most $2^{m-c\theta\sqrt{m}}$. 
\end{proof}

By the claim, for any subset $B'$ of $B_0$, %
\begin{align} \label{eq:s-small}
&|(A_0+B')\cap \{y:|y|_1 \in [m/2-\delta\sqrt{m},m/2]\}| \le |B'|2^{m-c\theta\sqrt{m}}.
\end{align}
Thus, noting that $B_0+B_0\subseteq \{y:|y|_1 \le m/2-\delta\sqrt{m}\}$, we have 
\begin{align} \label{eq:s-small-1}
&|(B_0+B_0)\cup (A_0+B')| \noindent\\
&\le |\{y:|y|_1 \le m/2-\delta\sqrt{m}\}| + |(A_0+B')\cap \{y:|y|_1 \in [m/2-\delta\sqrt{m},m/2]\}| \nonumber \\
&\le 2^m(|B'|2^{-c\theta\sqrt{m}} + \exp(-c_2\delta^2)),
\end{align}
for some constant $c_2$. 
For each $x\in \mathbb{F}_2^{p+m}$, write $x=(x_1,x_2)$ where $x_1 \in \mathbb{F}_2^{p}$ and $x_2 \in \mathbb{F}_2^{m}$. Consider $A\subseteq \mathbb{F}_2^{p+m}$ given by %
$$A((x_1,x_2))=\mathbb{I}(x_1=0)A_0(x_2) + (1-\mathbb{I}(x_1=0))B_0(x_2).$$ 
Here $\mathbb{I}$ denotes the indicator function%
, and for a set $S$, $S(x):=\mathbb{I}[x\in S]$. Then $A+A \subseteq \mathbb{F}_2^p\times \{x\in \mathbb{F}_2^m:|x|_1\le m/2\} \cup \{x:x_1=0\}$. %
The doubling constant of $A$ is $\lambda = |A+A|/|A| \le 2^{p+m}/(2^m \exp(-c\theta^2)) \le 2^{p}\exp(c\theta^2)$. 

For any subset $A'$ of $A$ of size at least $(1-\epsilon)|A|$, if $m$ is sufficiently large in terms of $\epsilon$, then $A'_0 := \{x_2: (0,x_2)\in A'\}$ has size at least $(1-2\epsilon)|A|$. Let $\Gamma_+=\Gamma_+^{1}(A)$ denote the set of points $y$ differing in at most one coordinate from some $x\in A$, and define $\Gamma_+^{k+1}(A)=\Gamma_+(\Gamma_+^{k}(A))$. Observe that $A'_0+B_0 = \Gamma_+^{\theta\sqrt{m}}(A'_0)$. Recall the following vertex isoperimetry inequality of Harper \cite{Har}.
\begin{thm}[\cite{Har}]\label{thm:iso}
Let $A\subseteq \{0,1\}^m$ have $|A| \ge \sum_{j=0}^{q}\binom{m}{j}$. Then $|\Gamma_+(A)| \ge \sum_{j=0}^{q+1}\binom{m}{j}$. 
\end{thm}
Note that 
\[
|A_0'| \ge (1-2\epsilon)|A_0| = (1-2\epsilon)\sum_{j=0}^{m/2-\theta \sqrt{m}} \binom{m}{j} \ge \sum_{j=0}^{m/2-(\theta+\eta) \sqrt{m}} \binom{m}{j},
\]
where $\eta\to 0$ as $\epsilon\to 0$. By Theorem \ref{thm:iso}, 
\[
|\Gamma_+^{\theta\sqrt{m}}(A_0')| \ge \sum_{j=0}^{m/2-\eta\sqrt{m}}\binom{m}{j} \ge 2^{m-1}\exp(-C\eta^2).
\]
Since $A'+A$ contains $2^{p}$ disjoint translates of $A'_0+B_0 = \Gamma_+^{\theta \sqrt{m}}(A'_0)$, we have for sufficiently small $\epsilon$ that 
\[
|A'+A| \ge 2^{p+m-1}\exp(-C\eta^2)\ge 2^{p+m-2}. 
\]

On the other hand, for any such $A'$ and any subset $A_{(s)}$ of $A$ of size at most $s$, let $B_0' =\{x_2:\: (x_1,x_2)\in A_{(s)}, \, x_1\neq 0\}$. Then $A'+A_{(s)}$ is contained in $\{x:x_1=0\} \cup (\mathbb{F}_2^{p}\times ((B_0+B_0)\cup (A'_0+B_0')))$. By (\ref{eq:s-small-1}) and the fact that $|B'_0|\le |A_{(s)}|\le s$, we have that $|A'+A_{(s)}|$ is bounded above by %
\[
|A'+A_{(s)}| \le 2^m + 2^{p+m}(\exp(-c_2\delta^2)+s2^{-c\theta\sqrt{m}}).
\]
We thus have that 
\[
|A'+A_{(s)}|/|A'+A| < 4(2^{-p}+\exp(-c_2\delta^2)+s2^{-c\theta \sqrt{m}}).
\]
Recall that we chose $p$ such that $2^p \nu > 100$. By choosing $\theta$ (and hence $\delta$) sufficiently large such that $\exp(-c_2\delta^2)<\nu/25$, as long as $m$ is sufficiently large such that $s<\nu 2^{c\theta\sqrt{m}-5}$, we conclude that $|A'+A_{(s)}|/|A'+A|<\nu$ as desired. 

\noindent \textbf{Construction over $\mathbb{Z}_N$.} Let $q$ be a prime so that $q\nu > 100$, $\theta>0$ sufficiently large, and $\theta/4<\delta<\theta/3$. Consider primes $p_1,\dots,p_m>q$. Let $N=p_1\dots p_m$. For each $i\le m$, let $A_i = \{2k+1:k\le p_i/2\} \subseteq \mathbb{Z}_{p_i}$. For each $x\in \prod_{i}\mathbb{Z}_{p_i}$, let $\pi(x) \in \{0,1\}^m$ be defined by $\pi(x)_i = A_i(x_i)$. Define $A_0 = \{x:|\pi(x)|_1 \le m/2-\theta \sqrt{m}\}$, where again $|\pi(x)|_1$ is the number of ones in $\pi(x)$. Again $|A_0|=N\exp(-c\theta^2)$. Let $B_0 = \{x\in \{0,1\}^m:|\pi(x)|_1 \le \theta\sqrt{m}\}$. %
As before, for any subset $B_{(s)}$ of $B_0$ of size $s$, 
\[
|(A_0+B_{(s)})\cap \{y:|\pi(y)|_1 \in [m/2-\delta\sqrt{m},m/2]\}| \le N s2^{-c\theta\sqrt{m}}.
\]

Consider the group $G = \mathbb{Z}_q \times \prod_i \mathbb{Z}_{p_i}$, and denote each element $z\in G$ by $z=(x',x)$ with $x'\in \mathbb{Z}_q$ and $x \in \prod_i \mathbb{Z}_{p_i}$. Define $A\subseteq \mathbb{Z}_q\times \prod_i \mathbb{Z}_{p_i}$ by 
\[
A((x',x))=\mathbb{I}(x'=0)A_0(x)+(1-\mathbb{I}(x'=0))B_0(x).
\]
Then $A+A \subseteq \mathbb{Z}_q\times \{x:|\pi(x)|_1\le m/2\} \cup \{z:x'=0\}$.  %
For any subset $A'$ of $A$ of size at least $(1-\epsilon)|A|$, as in the proof of the construction over $\mathbb{F}_2^{N}$, we have $|A'+A| \ge Nq/4$. For any such $A'$ and any subset $A_{(s)}$ of $A$ of size $s$, we can bound 
\[
|A'+A_{(s)}| \le |A+A_{(s)}| \le qN(q^{-1}+\exp(-c\delta^2)+s2^{-c\theta\sqrt{m}}) < \nu qN/4 \leq \nu |A'+A|.
\]  

\noindent \textbf{Construction over $\mathbb{Z}$ and $\mathbb{Z}_p$.} Let $A$ be the set in the construction over $\mathbb{Z}_N$ and note that upon choosing $q$ depending only on $\nu$, we can guarantee that $|A|\gg_{\nu} N$. Let $\hat{A}$ be the set of representatives of $A$ in $[N-1]$. Observe that $\hat{A}+\hat{A}\subseteq [2(N-1)]$. We have $|\hat{A}+\hat{A}|\le 2|A+A|\le 2\lambda |A|$. For any $\hat{A}'\subseteq \hat{A}$ of size $(1-\epsilon)|\hat{A}|$ and any $\hat{A}_{(s)}$ of size $s$, we have that 
\[
\left|\hat{A}'+\hat{A}\right|\ge \left|\hat{A}'+\hat{A} \, (\bmod N)\right| \ge \frac{1}{\nu}\left|\hat{A}'+\hat{A}_{(s)} \, (\bmod N)\right| \ge \frac{1}{2\nu}\left|\hat{A}'+\hat{A}_{(s)}\right|.
\]
For any prime $p > 2N$ with $p\ll_{\nu} N$, the same set $\hat{A}\pmod p$ gives a construction of a subset of $\mathbb{Z}_p$ with density $\Omega_{\nu}(1)$ satisfying the required properties in Theorem \ref{thm:saturating-sumset}.

\section{Concluding remarks}

We believe that the following question is interesting. 
\begin{ques}
Let $\lambda, \epsilon>0$ be positive constants. Let $A$ be a subset of an abelian group $G$ with $|A+A|\le \lambda|A|$. How large does a subset $A'$ of $A$ need to be to guarantee $|A'+A|\ge (1-\epsilon)|A+A|$? 
\end{ques}
As we saw in Section \ref{sec:all}, even if $\epsilon=0$, we have that $|A'|=o(|A|)$ suffices (with explicit power-saving bounds over abelian groups with bounded exponent, while Behrend's construction shows that one cannot have power-saving bounds over $\mathbb{Z}$ or $\mathbb{Z}_p$). On the other hand, for some fixed positive $\epsilon$, considering $A$ formed by taking the union of a large arithmetic progression and $\sqrt{|A|}$ arbitrary elements shows that one needs $|A'| = \Omega(\sqrt{|A|})$. It would be interesting to determine for $\epsilon>0$ if one needs $|A'|\ge |A|^{1-o(1)}$ when $G=\mathbb{Z}$ or $G=\mathbb{Z}_p$. 

Another interesting question is to determine the optimal exponent on the doubling constant in Theorem \ref{thm:sym}. 
\begin{ques}
What is the largest $C$ %
such that the following holds? For some constant $c>0$, if $A$ is a nonempty subset of an abelian group $G$ and $\kappa=|A+A|/|A|$, then for each $s\le \kappa^{C-o(1)}$, there exists $A_{(s)}\subseteq A$ of size at most $s$ such that $|A+A_{(s)}|\ge cs|A|$.
\end{ques}
As discussed previously, from Theorem~\ref{thm:sym} and Proposition~\ref{prop:lower-exponent} we know that $C\in [\frac{1}{3},\frac{1}{1.29}]$, while for the setting of difference sets $A-A$ we have the better lower bound $C\geq \frac{1}{2}$.

\end{document}